\documentclass[11pt,reqno]{amsart}

\usepackage{latexsym,bm, amsfonts,amsthm,amsmath,mathrsfs,CJK,amssymb}

\makeatletter

\newcommand{\Rmnum}[1]{\expandafter\@slowromancap\romannumeral #1@}
\makeatother

 \newtheorem{lem}{Lemma}[section]  \newtheorem{thm}{Theorem}[section]
 \newtheorem{defn}{Definition}[section]  
\numberwithin{equation}{section}

\numberwithin{equation}{section}

 \newcommand{\me}{\mathrm{e}} 
\newcommand{\dif}{\mathrm{d}} \DeclareMathAlphabet{\mathsfsl}{OT1}{cmss}{m}{sl} \DeclareMathAlphabet{\mathpzc}{OT1}{pzc}{m}{it}

    \newcommand{\ee}{\mathbb{E}}
   \newcommand{\nn}{\mathbb{N}} \newcommand{\rr}{\mathbb{R}}
    
\newcommand{\vv}{\mathbb{V}}
 \def\CC{\mathcal C}   \def\FF{\mathcal F}  \def\HH{\mathcal H} 

   \def\MM{\mathcal M}    
\def\PP{\mathscr P} \def\BB{\mathscr B}

 \def\d"{^{\prime\prime}} \def\d'{^{\prime}}

\begin{document}
		
		\title[]{Note on an intermediate Baum-Katz theorems under sub-linear expectations}
		\thanks{Supported by Doctoral Scientific Research Starting Foundation of Jingdezhen Ceramic University ( No. 102/01003002031), Academic Achievement Re-cultivation Projects of Jingdezhen Ceramic University (Nos. 215/20506341, 215/20506594, 215/20506277).}
		\date{} \maketitle
		\begin{center}
			
			MINGZHOU XU~\footnote{Email: mingzhouxu2022@163.com} AND
			JINJIN YANG
			\\
			School of Information Engineering, Jingdezhen Ceramic University, Jingdezhen 333403, China

		\end{center}

		{\bf Abstract}\quad		We study the Hsu-Robbins-Erd\"{o}s theorem to the case when all moments under sub-linear expectations exist, but the moment generating function does not,
        viz, we suppose that $C_{\vv}\left\{\exp\left\{\left(\log^{+}|X|\right)^{\alpha}\right\}\right\}<\infty$ for some $\alpha>1$. We also investigate multi-index version.	
		
		
		{\bf MR(2010) Subject Classification}\quad 60F15, 60G50
		
\section{Introduction}


        Peng \cite{peng2007g} gives  the concepts of the sub-linear expectations space to investigate the uncertainty in probability.  The results of Peng \cite{peng2007g} inspired people to study the results under sub-linear expectations space, which generalize 
        the corresponding ones in probability space. Zhang \cite{zhang2016exponential} studied exponential inequalities under sub-linear expectations. Under sub-linear expectations, 
        Xu and Cheng \cite{xu2022small} got how small the increments of $G$-Brownian motion are, Xu \cite{xu2025complete} studied complete $q$-order moment convergence of moving average processes generated by negatively random variables. 
        For more limit theorems under sub-linear expectations, the interested readers could refer to Gao and Xu \cite{gao2011large}, Peng \cite{peng2019nonlinear}, Zhang \cite{zhang2015donsker, zhang2016rosenthal, zhang2022strong}, 
        Chen \cite{chen2016strong}, Zhong and Wu \cite{zhong2017complete}, Wu and Jiang \cite{wu2018strong}, Xu and Zhang \cite{xu2019three, xu2020law},
        Kuczmaszewska \cite{kuczmaszewska2020complete}, Xu and Cheng \cite{xu2021convergence, xu2022note,xu2022noteb, xu2025equivalent}, Xu et al. \cite{xu2022complete, xu2023convergence}, 
        Xu and Kong \cite{xu2023noteon, xu2024completeqth}, Chen and Wu \cite{chen2022complete}, Xu \cite{xu2023acomplete, xu2023completeconvergencea,  xu2024onthe, xu2024sample}, Xu and Xie \cite{xu2025completemo}, and the references therein.
        
        In probability space, Gut and Stadtm\"{u}ller \cite{gut2011} got an intermediate Baum-Katz theorem. For references on complete convergence in linear expectation space, 
        the interested reader could refer to Wang et.al \cite{wang2024some} and the refercences therein. Motivated by the works of
        Gut and Stadtm\"{u}ller \cite{gut2011}, Gao and Xu \cite{gao2011large}, we try to investigate 
        an intermediate Baum-Katz theorem  under sub-linear expectations, which complements that in Gut and Stadtm\"{u}ller \cite{gut2011}.
        
        The rest of this paper is presented as follows. We give necessary notions, concepts, relevant  properties, and lemmas under sub-linear expectations in the next section. 
        In Section 3, we present our main result, Theorem \ref{thm01},  the proof of which is postponed in Section 4. 
        In Section 5, we present the multi-index version of Theorem \ref{thm01}, Theorem \ref{thm02}, and explain its proof.
		
\section{Preliminary}
	\setcounter{equation}{0}
	We use similar notations as in the work by Peng \cite{peng2019nonlinear}, Zhang \cite{zhang2016exponential}. 
	Assume that $(\Omega,\FF)$ is a given measurable space. Suppose that $\HH$ is a subset of all random variables 
	on $(\Omega,\FF)$ satisfying that $X_1,\cdots,X_n\in \HH$ yields $\varphi(X_1,\cdots,X_n)\in \HH$ for every $\varphi\in \CC_{l,Lip}(\rr^n)$, where $\CC_{l,Lip}(\rr^n)$ stands for the linear space of (local lipschitz) function $\varphi$ fulfilling
	$$
	|\varphi(\mathbf{x})-\varphi(\mathbf{y})|\le C(1+|\mathbf{x}|^m+|\mathbf{y}|^m)(|\mathbf{x}-\mathbf{y}|), \forall \mathbf{x},\mathbf{y}\in \rr^n
	$$
	for some $C>0$, $m\in \nn$ relying on $\varphi$.
	\begin{defn}\label{defn1} A sub-linear expectation $\ee$ on $\HH$ is a functional $\ee:\HH\mapsto \bar{\rr}:=[-\infty,\infty]$ fulfilling the properties below: for all $X,Y\in \HH$, we have
		\begin{description}
			\item[\rm (a)] If $X\ge Y$, then $\ee[X]\ge \ee[Y]$;
			\item[\rm (b)] $\ee[c]=c$, $\forall c\in\rr$;
			\item[\rm (c)] $\ee[\lambda X]=\lambda\ee[X]$, $\forall \lambda\ge 0$;
			\item[\rm (d)] $\ee[X+Y]\le \ee[X]+\ee[Y]$ whenever $\ee[X]+\ee[Y]$ is not of the form $\infty-\infty$ or $-\infty+\infty$.
		\end{description}
		
	\end{defn}
	We call $V:\FF\mapsto[0,1]$ to be a capacity if
	\begin{description}
		\item[\rm (a)]$V(\emptyset)=0$, $V(\Omega)=1$;
		\item[\rm (b)]$V(A)\le V(B)$, $A\subset B$, $A,B\in \FF$.
	\end{description}
	A capacity $V$ is named to be sub-additive if $V(A\bigcup B)\le V(A)+V(B)$, $A,B\in \FF$.

	In this sequel, given a sub-linear expectation space $(\Omega, \HH, \ee)$, we suppose that $\ee$ 
	is regular, i. e., for all $\{X_n, n\in \nn\}\subset\HH$, $X_n(\omega)\downarrow0$, 
	$\omega\in\Omega\Longrightarrow\lim_{n\rightarrow\infty}\ee(X_n)=0$. Let $\MM$ stands for the collections
	of all probability measures on $\Omega$. By Theorem 1.2.2 of Peng \cite{peng2019nonlinear}, there
	exists a relatively compactly subset $\PP\subset \MM$ satisfying that 
	\[\ee(X)=\sup_{P\in \PP}\mathbf{E}_P(X), \mbox{  for all $X\in \HH$,}
	\]
	where $\mathbf{E}_P(X)$ means the linear expectation of $X$ under probability measure $P$.
	For any Borel measurable function $X$, set 
	\[\bar{\ee}(X)=\sup_{P\in\PP}\mathbf{E}_P(X), \vv(A)=\bar{\ee}(I_A), \forall A\in \BB(\Omega),
	\]
	where $I(A)$ or $I_A$ stands for the indicator function of $A$ in this sequel.
	Then $\vv$ is a sub-additive capacity. 
	Set
	$$
	C_{\vv}(X):=\int_{0}^{\infty}\vv(X>x)\dif x +\int_{-\infty}^{0}(\vv(X>x)-1)\dif x.
	$$
	
	Assume that $\mathbf{X}=(X_1,\ldots,X_m)$, $X_i\in \HH$ and $\mathbf{Y}=(Y_1,\ldots,Y_n)$, $Y_i\in \HH$ are two random variables on $(\Omega,\HH,\ee)$. $Y$ is said to be independent of $X$, if each
	$\psi\in \CC_{l,Lip(\rr^{m+n})}$, we have
	\[\ee[\pi(\mathbf{X},\mathbf{Y})]=\ee\left[\ee\left[\psi(\mathbf{x},\mathbf{Y})\right]|_{\mathbf{x}=\mathbf{X}}\right],
	\]
	cf. Definition 1.3.11 of Peng \cite{peng2019nonlinear}. $\{X_n\}_{n=1}^{\infty}$ is called to be
	a sequence of independent random variables, if $X_{n+1}$ is independent of $(X_1,\ldots,X_n)$ for each $n\ge 1$.
	
    Assume that $\mathbf{X}_1$ and $\mathbf{X}_2$ are two $n$-dimensional random vectors defined, respectively, in sub-linear expectation spaces $(\Omega_1,\HH_1,\ee_1)$ and $(\Omega_2,\HH_2,\ee_2)$. They are said to be identically distributed if  
	for every $\psi\in \CC_{l,Lip}(\rr^n)$ such that $\psi(\mathbf{X}_1)\in \HH_1, \psi(\mathbf{X}_2)\in \HH_2$,
	$$
	\ee_1[\psi(\mathbf{X}_1)]=\ee_2[\psi(\mathbf{X}_2)], \mbox{  }
	$$
	whenever the sub-linear expectations are finite. $\{X_n\}_{n=1}^{\infty}$ is said to be identically distributed if for each $i\ge 1$, $X_i$ and $X_1$ are identically distributed.
	
	We give two lemmas below.
	\begin{lem}\label{lem01}
		Suppose that $X\le 1$ and $\ee(X)\le 0$ in the sub-linear expectation space $(\Omega, \HH, \ee)$.
		Then we have
	\begin{align*}
		\ee\left(\exp(X)\right)\le \exp(\ee(X^2)).
	\end{align*}
	\end{lem}
	\begin{proof}
	By the fact that $\frac{\exp(x)-1-x}{x^2}\ge 0$ is a nondecreasing function of $x$, we see that
	\[\exp(X)-1-X\le X^2(\me-1-1),\]
	which together with $\ee(X)\le 0$ yields that
	\[\ee(\exp(X))\le \ee(1+X+X^2)\le 1+\ee(X^2)\le \exp(\ee(X^2)).\]
	This finishes the proof.
	\end{proof}
	\begin{lem}\label{lem02}
		Assume that $\{X_n,n\ge 1\}$ is a sequence of independent random variables in $(\Omega, \HH, \ee)$. Then for any open sets $O_k\in \BB(\rr)$, 
        any closed sets $F_k\in \BB(\rr)$, $k=1,2,\ldots$, $n\ge 2$, we have
        \begin{align*}
            \vv\left(\bigcap_{k=1}^{n}\left\{X_k\in O_k\right\}\right)=\prod_{k=1}^{n}\vv\left(X_k\in O_k\right),
        \end{align*}
        \begin{align*}
            \vv\left(\bigcap_{k=1}^{n}\left\{X_k\in F_k\right\}\right)=\prod_{k=1}^{n}\vv\left(X_k\in F_k\right).
        \end{align*}
	\end{lem}
	\begin{proof}By setting $\varphi_k=I_{O_k}$ and $\psi_k=I_{F_k}$, $k\ge 1$ in Lemma 1.1 of Gao and Xu \cite{gao2011large}, we finish the proof.
	\end{proof}
	
	In the paper we suppose that $\ee$ is countably sub-additive, i.e., $\ee(X)\le \sum_{n=1}^{\infty}\ee(X_n)$, whenever $X\le \sum_{n=1}^{\infty}X_n$, $X,X_n\in \HH$, and $X\ge 0$, $X_n\ge 0$, $n=1,2,\ldots$. 
	Let $C$  represents positive constant which may change from place to place. $a\bigvee b$ stands for the maximum of $a$ and $b$.

\section{Main results}
		\setcounter{equation}{0}
		Our main result is as below.
\begin{thm}\label{thm01}
		Let $\alpha>1$, and assume that $X, X_1, X_2,\ldots$ are independent, identically distributed (i. i. d.) random variables with partial sums 
		$S_n=\sum_{k=1}^{n}X_k$, $n\ge 1$. Suppose that 
		\begin{equation}\label{eq01}
		\ee\left\{\exp\left\{\left(\log^{+}|X|\right)^{\alpha}\right\}\right\}\le C_{\vv}\left\{\exp\left\{\left(\log^{+}|X|\right)^{\alpha}\right\}\right\}<\infty,\ 
		\mbox{and}\  \ee(X)=\ee(-X)=0. 
		\end{equation}
		Then
\begin{equation}\label{eq02}
\sum_{n=1}^{\infty}\exp\left\{(\log n)^{\alpha}\right\}\frac{(\log n)^{\alpha-1}}{n^2}\vv\left\{|S_n|>n\varepsilon\right\}<\infty \ 
\mbox{ for all $\varepsilon>1$.}
\end{equation}

\end{thm}

\section{Proof of Theorem \ref{thm01}}
\begin{proof}
(\ref{eq01})$\Rightarrow$(\ref{eq02}): Let $0<\delta<1$ and $\varepsilon>0$ be arbitrary, write, for $n\ge 1$,
\begin{align}\label{05}
b_n =\frac{\varepsilon n}{2(\log n)^{\alpha}}, \mbox{ and } c_n=n\varepsilon(1-\delta).
\end{align}
Set, for $1\le k\le n$,
\begin{align}
&X_k'=X_kI\left\{|X_k|\le b_n\right\}+b_nI\{X_k>b_n\}-b_nI\{X_k<-b_n\}\nonumber\\
&X_k''= (X_k-b_n)I\left\{b_n<X_k<c_n\right\}+(c_n-b_n)I\left\{X_k\ge c_n\right\}\nonumber\\
&\quad \quad+(X_k+b_n)I\left\{-c_n<X_k<-b_n\right\}+(-c_n+b_n)I\left\{X_k\le -c_n\right\}\nonumber\\
&X_k'''=(X_k-c_n)I\{X_k>c_n\}+(c_n+X_k)I\{X_k<-c_n\}\nonumber
\end{align}
and $X_k=X_k'+X_k''+X_k'''$, and let all objects with primes and multiple primes refer to the respective truncated summands (throughout this paper 
$\log n=\log^{+}n$).

Next, write $A_n=\{|S_n|> n\varepsilon\}$,
\begin{align*}
&A_n'=\left\{|S_n|>n\varepsilon \mbox{  and $X_k''\not=0$ for at most one $k\le n$ and $X_k'''=0$ for all $k\le n$}\right\}\\
&A_n''=\left\{X_k''\not=0 \mbox{  for at least two $k\le n$}\right\},\\
&A_n'''=\left\{X_k'''\not=0 \mbox{ at least one $k\le n$}\right\}.
\end{align*}
We moreover split $A_n'$ into $A_{n,1}'\bigcup A_{n,2}'$, where
\begin{align*}
&A_{n,1}'=\left\{|S_n|>n\varepsilon \mbox{ and $X_k\le b_n$ for all $k\le n$}\right\},\\
&A_{n,2}'=\left\{|S_n|>n\varepsilon \mbox{ and $X_k>b_n$ for exactly one $k\le n$}\right\},
\end{align*}
and obeserve that 
\begin{equation}\label{06}
A_n\subset A_{n,1}'\bigcup A_{n,2}'\bigcup A_n''\bigcup A_n''',
\end{equation}
which implies that 
\begin{equation}\label{07}
	\vv\left\{|S_n|>n\varepsilon\right\}=\vv(A_n)\le \vv(A_{n,1}')+\vv(A_{n,2}')+\vv(A_n'')+\vv(A_n''')\end{equation}

$\bullet \vv(A_{n,1}')$\\
Since truncation destroys centering, we see that
\begin{align*}
|\ee(S_n')|\bigvee|\ee(-S_n')|&=|n\ee(X_k')|\bigvee|n\ee(-X_k')|=n\left(\left|\ee(X_k'-X_k)\right|\bigvee\left|\ee(-X_k'+X_k)\right|\right)\\
&\le n\ee|X-X'|\le nC_{\vv}\left\{|X|I\left\{|X|>b_n\right\}\right\}\\
&\le \frac{nC_{\vv}\left\{|X|^2I\left\{|X|>b_n\right\}\right\}}{b_n}=\frac{2(\log n)^{\alpha}}{\varepsilon}C_{\vv}\left\{|X|^2I\left\{|X|>b_n\right\}\right\}\\
&=o\left((\log n)^{\alpha}\right) \mbox{ as $n\rightarrow\infty$.}
\end{align*}
Hence, by the exponential bound and Lemma \ref{lem01}, we get, for $n\ge n_0$ large,
\begin{align}\label{08}
	\vv\left\{A_{n,1}'\right\}&=\vv\left\{S_n'>n\varepsilon\right\}+\vv\left\{-S_n'>n\varepsilon\right\}\nonumber\\
	&\le \vv\left\{S_n'-\ee S_n'>(n-\delta(\log n)^{\alpha}\varepsilon)\right\}+\vv\left\{-S_n'-\ee (-S_n')>(n-\delta(\log n)^{\alpha}\varepsilon)\right\}\nonumber\\
	&\le C\exp\left\{-\frac{(\log n)^{\alpha}}{n\varepsilon}(n-\delta(\log n)^{\alpha}\varepsilon)+\frac{(\log n)^{2\alpha}}{n^2\varepsilon^2}n\ee(X_1^2)\right\}\nonumber\\
	&\le C\exp\left\{-(\log n)^{\alpha}\right\},
\end{align}
and therefore that
\begin{align}\label{09}
	\sum_{n\ge n_0}\exp\left((\log n)^{\alpha}\right)\frac{(\log n)^{\alpha-1}}{n^2}\vv\left(A_{n,1}'\right)\le C\sum_{n\ge n_0}\frac{(\log n)^{\alpha-1}}{n^2}<\infty.
\end{align}

$\bullet \vv(A_{n,2}')$\\
First obeserve that 
\begin{align*}
n\vv\left\{|X|>b_n\right\}\le n\frac{\ee\left(\exp\left\{\left(\log^{+}|X|\right)^{\alpha}\right\}\right)}{\exp\left\{(\log b_n)^{\alpha}\right\}},
\end{align*}
which, together with the fact that for large $n_1$, $n\ge n_1$,
\begin{align*}
(\log b_n)^{\alpha}=\left(\log\left(\varepsilon+\log n-\alpha\log\log n-\log 2\right)\right)^{\alpha}\ge (1-\delta/2)(\log n)^{\alpha},
\end{align*}
yields that
\begin{align}\label{10}
	n\vv\left\{|X|>b_n\right\}\le Cn\exp\left(- (1-\delta/2)(\log n)^{\alpha}\right)\rightarrow 0 \mbox{  as $n\rightarrow\infty$.}
\end{align}
Next, by Lemma \ref{lem02}, (\ref{08}) and (\ref{10}), we get, for $n\ge n_1$,
\begin{align*}
	\vv(A_{n,2}')&\le \vv\left\{|S_{n-1}'|>\varepsilon n-c_n\right\}n\vv\left\{|X|>b_n\right\}\\
	&\le \vv\left\{|S_{n-1}'|>\varepsilon\delta n\right\}n\vv\left\{|X|>b_n\right\}\le Cn\exp\left\{-\delta(\log n)^{\alpha}-(1-\delta/2)(\log n)^{\alpha}\right\},
\end{align*}
which yields that
\begin{align*}
	\sum_{n\ge n_1}\exp\left((\log n)^{\alpha}\right)\frac{(\log n)^{\alpha-1}}{n^2}\vv\left(A_{n,2}'\right)\le C\sum_{n\ge n_1}\frac{(\log n)^{\alpha-1}}{n^{(1+\delta/2)}}<\infty.
\end{align*}

$\bullet \vv(A_{n}'')$\\
By (\ref{10}) and Lemma \ref{lem02}, we see that
\begin{align}\label{11}
	\sum_{n\ge n_1}\exp\left((\log n)^{\alpha}\right)\frac{(\log n)^{\alpha-1}}{n^2}\vv\left(A_{n}''\right)\le C\sum_{n\ge n_1}\exp\left\{-(1-\delta)(\log n)^{\alpha}\right\}(\log n)^{\alpha}. 
\end{align}

$\bullet \vv(A_n''')$\\
Since 
\begin{align*}
\vv(A_n''')\le n\vv\left\{|X|>c_n\right\}=n\vv\left\{|X|>n\varepsilon(1-\delta)\right\},
\end{align*}
we conclude that
\begin{align*}
	\sum_{n\ge n_1}\exp\left((\log n)^{\alpha}\right)\frac{(\log n)^{\alpha-1}}{n^2}\vv\left(A_{n}'''\right)\le \sum_{n\ge n_1}\exp\left((\log n)^{\alpha}\right)\frac{(\log n)^{\alpha-1}}{n}\vv\left\{|X|>n\varepsilon(1-\delta)\right\}
\end{align*}
	and the latter sum converges iff $\varepsilon(1-\delta)\ge 1$ by Lemma \ref{lem2.1} below.
\end{proof}	

\begin{lem}\label{lem2.1}
For any random variable $X$ and $\gamma>0$,
\begin{align*}
	C_{\vv}\left\{\exp\left\{\left(\log^{+}(X/\gamma)\right)^{\alpha}\right\}\right\}<\infty\Longleftrightarrow \sum_{n=1}^{\infty}\exp\left\{(\log n)^{\alpha}\right\}\frac{(\log n)^{\alpha-1}}{n}\vv\left(|X|>n\gamma\right)<\infty.
\end{align*}
\end{lem}
\begin{proof}
The proof is similar to that of Lemma 2.2 of Zhong and Wu \cite{zhong2017complete}, and hence is omitted here.
\end{proof}
\section{Random fields}
In order to describe the scene, let $\mathbf{Z}_{+}^d$, $d\ge 2$, denote the positive integer $d$-dimensional 
lattice with coordinate-wise partial ordering $\le$, that is, for $\mathbf{m}=(m_1,m_2,\ldots,m_d)$ and $\mathbf{n}=(n_1,n_2,\ldots,n_d)$,
 $\mathbf{m}\le \mathbf{n}$ means that $m_k\le n_k$, for $k=1,2,\ldots, d$. The size of a poinit means $|\mathbf{n}|=\prod_{k=1}^{d}n_k$, and 
 $\mathbf{n}\rightarrow\infty$ represents that $n_k\rightarrow\infty$, for all  $k=1,2,\ldots, d$.
In the following Theorem \ref{thm01} is extended to multi-index models or random fields.
\begin{thm}\label{thm02}
Let $\alpha>1$, and assume that $\{X_{\mathbf{k}},\mathbf{k}\in \mathbf{Z}_{+}^d\}$ are i. i. d. random variables with $\ee(X)=\ee(-X)=0$ and 
partial sums $S_{\mathbf{n}}=\sum_{\mathbf{k}\le \mathbf{n}}X_{\mathbf{k}}$. If 
\begin{align*}
	C_{\vv}\left\{\exp\left\{\left(\log|X|\right)^{\alpha}
	\right\}\left(\log^{+}|X|\right)^{d-1}\right\}<\infty
\end{align*}
holds, then we have
\begin{align*}
\sum_{\mathbf{n}}\exp\left\{(\log |\mathbf{n}|)^{\alpha}\right\}\frac{(\log |\mathbf{n}|)^{\alpha-1}}{|\mathbf{n}|^2}\vv\left(|S_{\mathbf{n}}|>|\mathbf{n}|\varepsilon\right)<\infty \mbox{ for all $\varepsilon>1$.}
\end{align*}
\end{thm}
\begin{proof}
The proof is similar  adaptation of those in Gut \cite{gut1978} combined with that of Theorem \ref{thm01} and Lemma \ref{lem3.1} below, and hence is omitted here.
\end{proof}
\begin{lem}\label{lem3.1}
	If for any random variables $\gamma>0$,  
\begin{align*}
	C_{\vv}\left\{\exp\left\{\left(\log^{+}|X/\gamma|\right)^{\alpha}\right\}\left(\log^{+}|X|\right)^{d-1}\right\}<\infty,
\end{align*}
holds, then we have
\begin{align*}
	\sum_{\mathbf{n}}\exp\left\{(\log|\mathbf{n}|)^{\alpha}\right\}\frac{(\log|\mathbf{n}|)^{\alpha-1}}{|\mathbf{n}|}\vv\left(|X|>|\mathbf{n}|\gamma\right)<\infty.
\end{align*}
\end{lem}
\begin{proof}
By similar proof of Lemma 2.2 of Gut \cite{gut1978} and  that of Lemma 2.2 of Zhong and Wu \cite{zhong2017complete}, we conclude this result. Hence, the proof is omitted here.
\end{proof}


\end{document}